\documentclass[11pt]{amsart}
\usepackage{url}
\usepackage{hyperref}
\usepackage{xcolor}
\usepackage{comment}
\usepackage{amssymb}

\newcommand{\F}{\mathbb{F}}

\newcommand{\E}{\mathbb{E}}
\newcommand{\Z}{\mathbb{Z}}
\newcommand{\Q}{\mathbb{Q}}

\newcommand{\GL}{\mathrm{GL}}

\newcommand{\PGL}{\mathrm{PGL}}

\renewcommand{\epsilon}{\varepsilon}

\renewcommand{\det}{\mathrm{det}}
\newcommand{\Ddet}{\mathrm{Det}}

\newcommand{\sut}{\hspace{2.5pt} | \hspace{2.5pt}}

\newtheorem{theorem}{Theorem}[section]
\newtheorem{lemma}[theorem]{Lemma}

\newtheorem{corollary}[theorem]{Corollary}

\newtheorem{notn}[theorem]{Notation}

\theoremstyle{definition}

\theoremstyle{definition}
\newtheorem{remark}[theorem]{Remark}

\theoremstyle{definition}

\begin{document}

\title{Locally nilpotent linear groups}

\begin{abstract}
We survey aspects of 
locally nilpotent linear groups. Then 
we obtain a new classification; namely, we classify 
the irreducible maximal 
locally nilpotent subgroups of $\GL(q, \F)$
for prime $q$ and any field $\F$.
\end{abstract}

\author{A.~S. Detinko and D.~L. Flannery}

\maketitle

\section{Why locally nilpotent linear groups?}

Linear (matrix) groups are commonly used to
represent abstract groups. Early work 
on linear groups was undertaken in the second 
half of the nineteenth century.  
Linear group theory is now a rich 
subject deeply connected to other areas of 
mathematics. Over the past few decades, revived 
interest in 
matrix groups has been driven 
partly by the development of computational 
group theory.

Recall that a group is said to be locally 
nilpotent if every finite
subset generates a nilpotent subgroup.
Locally nilpotent groups 
therefore generalize nilpotent groups. 
Many structural and classification results 
for locally nilpotent linear groups are known. 
Further progress can be made via 
computational techniques.

Group-theoretic algorithms accept
a finite generating set  as input. 
The celebrated `Tits alternative' states that a 
finitely generated matrix group 
either is solvable-by-finite (i.e., it has a 
normal solvable subgroup of finite 
index), or it contains a nonabelian free 
subgroup. For groups of the latter kind, 
basic computational problems, such as 
membership testing and the conjugacy problem, are
undecidable in general. On the other hand, 
locally nilpotent linear 
groups, being solvable and nilpotent-by-finite,
are more tractable for computation. 
(Note that a solvable linear group is 
nilpotent-by-abelian-by-finite.) This point is 
accentuated by Gromov's result~\cite{Gromov}, 
which implies that a finitely generated group has 
polynomial growth if and only if it is 
nilpotent-by-finite.
Hence, as explained in \cite{Beals}, certain 
algorithmic efficiency problems can be 
overcome for locally nilpotent linear groups.

Extra motivation for continued study of locally 
nilpotent linear groups arises from their 
applications. 
Here we mention almost 
crystallographic groups, which appear as 
nilpotent-by-finite linear groups over 
$\Q$~\cite[$\S \hspace{.5pt} 8.2.3$]{Holt}. 
Moreover, a finitely generated nilpotent group 
is polycyclic and so isomorphic to a
subgroup of $\GL(n, \Z)$ for some $n$.  
Algorithms for nilpotent subgroups of $\GL(n, \Q)$
thereby serve as a key step toward algorithms for 
abstract finitely generated nilpotent groups.

\section{The structure of locally nilpotent linear groups}

A systematic investigation of 
locally nilpotent linear groups
was carried out by D.~A.~Suprunenko, starting in 
the late 1940s~\cite{SuprPre}. He classified 
the maximal locally nilpotent subgroups of $\GL(n,\F)$ 
for algebraically closed $\F$. Several authors 
 extended Suprunenko's results to other $\F$. In 
particular, criteria for the number 
of $\GL(n,\F)$-conjugacy 
classes of maximal locally nilpotent subgroups
to be finite, and classification in 
partial cases, were enabled by a detailed 
description of locally nilpotent linear groups 
over an arbitrary field $\F$
(see, e.g., \cite{Konyukh8}).

In this section, we outline  some of the most
important structural results about locally nilpotent 
linear groups. Much research has focused on the 
maximals, due to the fact that 
each locally nilpotent subgroup of 
$\GL(n,\F)$ is contained in a maximal locally nilpotent
subgroup (by contrast, a nilpotent 
subgroup of $\GL(n,\F)$ might not
be contained in a maximal nilpotent subgroup). 
We follow a standard reduction
scheme: reducible $\rightarrow$ completely reducible 
$\rightarrow$ irreducible $\rightarrow$ absolutely irreducible 
$\rightarrow$ primitive.

\subsection{Reducible locally nilpotent linear groups} 
We use terminology for linear groups as in 
\cite{Supr,Wehrfritz}. 
A reducible subgroup $G$ of $\GL(n, \F)$ is conjugate 
to a group of block upper triangular matrices, where
the main diagonal blocks (\emph{irreducible
parts}) give irreducible representations 
of $G$ over $\F$ in smaller degree. Suppose that
 $G$ is locally nilpotent and indecomposable.
Then the irreducible parts of $G$ are pairwise
equivalent~\cite[p. \hspace{-3pt}223, Theorem~2]{Supr}, 
so that $G$ may be conjugated
to a group of block upper triangular matrices
\[
\qquad \begin{bmatrix}
{\tiny a(g)} & a_{12}(g)  & \cdots & a_{1k}(g) \\

& & \vspace{-10pt}\\

0 & {\tiny a(g)}  & \cdots & a_{2k}(g) \\

& & \vspace{-10pt}\\

\vdots & \vdots & \ddots & \vdots\\

& & \vspace{-10pt}\\

0 & 0 & \cdots & {\tiny a(g)} 
\end{bmatrix}, \quad g\in G
\]
where $a_{ij}(g)\in \mathrm{Mat}(n/k,\F)$
and $a(G) = \{a(g) \sut g \in G\} \leq \GL(n/k, \F)$ 
is irreducible (and locally nilpotent). If 
$\F$ is perfect then $a_{ij} (g) = c_{ij} (g)a(g)$ where 
$c_{ij} (g)$ centralizes $a(G)$; 
thus $G$ is contained in the direct product of
a unitriangular group over a division algebra and 
a completely reducible group over $\F$ with equivalent 
irreducible parts. This reduces study of locally 
nilpotent linear groups to the irreducible case.

Another passage to the completely 
reducible case uses the Jordan decomposition. 
For each $g \in \GL(n,\F)$ there are a unique 
unipotent matrix $g_u \in \allowbreak 
\GL(n,\overline{\F}\hspace{.5pt} )$
and a unique diagonalizable matrix 
$g_d  \in \GL(n,\overline{\F}\hspace{.5pt} )$ such 
that $g = g_d g_u = g_u g_d$ ($\hspace{.6pt}\overline{\F}$ 
 denoting the algebraic closure of $\F\hspace{.6pt}$); 
see \cite[7.2]{Wehrfritz}. 
If $\F$ is perfect then $g_u, g_d\in \GL(n,\F)$.
\begin{theorem}[{\cite[p. \hspace{-3pt}240, Theorem~6]{Supr}} 
and {\cite [7.11]{Wehrfritz}}]
\label{Thm21}
Let $G \leq \GL(n, \F)$ be locally
nilpotent. Define 
$G_u = \{g \in G \sut g \hspace{6pt} \mathrm{unipotent}\}$, 
$G_d = \{g \in G \sut g \hspace{6pt} \mathrm{diagonalizable}\}$.
Then $G_u,  G_d \hspace{.5pt} \unlhd \hspace{.5pt} G$ and 
$\langle G_u , G_d\hspace{.2pt} \rangle = G_u \times G_d$.
\end{theorem}

By Theorem~\ref{Thm21}, if $G$ is completely 
reducible locally nilpotent then every subgroup 
of $G$ is completely reducible; so the
elements of $G$ are diagonalizable 
(see \cite[p. \hspace{-6.5pt} 239, Theorem~5]{Supr} 
and \cite[7.12]{Wehrfritz}).
\begin{corollary}[{\cite[7.13]{Wehrfritz}}]
If $G \leq \GL(n,\F)$ is locally nilpotent and splittable
(meaning that $g_u , g_d \in G$ for each $g \in G$), 
then $G = G_u \times  G_d$.
\end{corollary}

For example, if $\F$ is finite and $G$ is 
nilpotent then $G$ is splittable.
More generally:
\begin{theorem}[{\cite[p. \hspace{-6.5pt} 136, 
Proposition~3]{Segal}}] 
\label{Thm23}
Let $G$ be a nilpotent subgroup of 
$\hspace{.5pt} \GL(n,\F)$, where $\F$ is 
a perfect field. Then
$\widehat{G}_u := \{g_u \sut g \in G\}$,
  $\widehat{G}_d := \{g_d \sut g \in G\}$
 are subgroups of $\GL(n,\F)$,
$\widehat{G}_d$ is completely reducible, 
and $G\leq \widehat{G}_u \times \widehat{G}_d$.
\end{theorem}

\subsection{Irreducible locally 
nilpotent linear groups} 
An irreducible maximal locally
nilpotent subgroup of $\GL(n, \F)$ can be 
thought of as an absolutely irreducible 
maximal locally nilpotent subgroup of 
$\GL(m, \E)$ for some $m$ dividing $n$ and field 
$\E \supseteq \F$ 
(see \cite[p. \hspace{-3pt}217, Theorem~4]{Supr}). 
This affords a reduction to the absolutely 
irreducible case, particularly in the 
classification of irreducible maximal locally 
nilpotent subgroups of $\GL(n,\F)$. Subsequent reduction 
is possible in two directions, which are not 
mutually exclusive: to $p$-subgroups 
of $\PGL(n, \F)$, and to primitive groups. 
The first of these possibilities is based on 
the next result
($1_m$ denotes the $m\times m$ identity matrix).
\begin{theorem}[{\cite[pp. \hspace{-6pt} 220--221]{Supr}}]
\label{Thm24}
Let $n = p_1^{a_1} \cdots p_k^{a_k}$  
where the $p_i$ are distinct primes.
\begin{itemize}
\item[{\rm (i)}]  
 If $G$ is a 
maximal absolutely irreducible locally nilpotent 
subgroup of $\GL(n,\F)$, then 
$G = G_1\otimes \cdots \otimes G_k$ where 
$G_i \leq \GL(p_i^{a_i} , \F)$ for $1 \leq i \leq k$
 is maximal absolutely
irreducible locally nilpotent.
\item[{\rm (ii)}] 
Let $k=1$, and suppose that
$G$ is an absolutely irreducible subgroup of 
$\GL(n , \F)$ containing 
$\F^{\times} 1_{n}$. Then $G$ is a maximal absolutely 
irreducible locally nilpotent subgroup of 
$\GL(n, \F)$ if and only if 
$G/\F^{\times} 1_{n}$ is a Sylow $p$-subgroup 
of $\mathrm{PGL}(n , \F)$.
\end{itemize}
\end{theorem}

An irreducible locally nilpotent linear group is
center-by-periodic; indeed, its central
quotient is a direct product of $p$-groups 
(see \cite[Corollary~3.2.4]{ShirvaniWehrfritz}).

Except when $\F$ is finite or algebraically closed, 
the description of Sylow $p$-subgroups
of $\mathrm{PGL}(n,\F)$ is quite different from the 
description of Sylow $p$-subgroups of $\GL(n,\F)$.
The former were considered in \cite{KonyukhPGL}, 
mainly for $p > 2$. 
Classifying the Sylow $2$-subgroups of 
$\mathrm{PGL}(n,\F)$  is difficult. In 
\cite{KonyukhPGL}, $p$-subgroups of 
$\mathrm{PGL}(n,\F)$ are handled by the same 
techniques as those used for locally nilpotent 
linear groups. 

The reduction to primitives is not as straightforward 
for locally nilpotent linear groups as it is for 
other types of linear groups.
To appreciate this disparity, note that an 
irreducible imprimitive solvable subgroup 
of $\GL(n,\F)$ is conjugate
to a subgroup of $G \hspace{.5pt} \wr \hspace{.5pt} T$ 
where $G\leq \GL(m,\F)$ is primitive solvable and
$T$ is a transitive solvable permutation group of 
degree $n/m$~\cite[p. \hspace{-3pt}129, Theorem~5]{Supr}; 
however, the wreath product of
a locally nilpotent linear group and a nilpotent 
permutation group need not even be 
(locally) nilpotent.

As in \cite[$\S  \hspace{.5pt} 2$]{Konyukh8}, nilpotent
 primitive groups $G$ may be treated 
using the series $G \geq H \geq  K \geq 1$, where
$K = [G, G]$ and $H = \mathsf{C}_G (K)$. 
Some basic information follows.
\begin{theorem}[{\cite[Theorem~2]{Konyukh8}}]
Let $n$ be a power of a 
prime $p\neq \mathrm{char}\hspace{3pt} \F$,
and let $G$ be a primitive absolutely 
irreducible locally nilpotent subgroup 
 of $\GL(n ,\F)$.
Then $K$ is an abelian $p$-group, 
$\Sigma = \langle K\rangle_\F$ is a field, 
$G/H \cong \mathrm{Gal}(\Sigma/\F)$, 
$[H, H] \leq  \F^\times 1_{n}$, 
and $H$ is a primitive absolutely irreducible 
class $2$ nilpotent subgroup of $\GL(m, \Sigma)$,
$m = \allowbreak n /|\Sigma : \F|$.
\end{theorem}
We remark that the primitive nilpotent 
linear groups over finite fields 
are classified in \cite{Glasgow}.

The paper \cite{Konyukh8} also
gives methods to classify maximal locally
nilpotent subgroups of $\GL(n,\F)$ for arbitrary 
$\F$. Special attention 
is paid to the problem of determining when 
the number of $\GL(n,\F)$-conjugacy classes
of these subgroups is finite. This
depends on finiteness of the groups 
$\F^\times /(\F^\times )^m$ for $m$ dividing $n$.
 
Groups over an algebraically closed field
have been the most intensively examined.
\begin{theorem}[{\cite[Chapter III]{SuprPre}, 
\cite[Chapter~VII]{Supr}}]
\label{SuprPreThm}
Let $\F$ be algebraically closed.
\begin{itemize}
\item[{\rm (i)}] Irreducible locally nilpotent 
subgroups of $\GL(n,\F)$ exist if and only 
if $\mathrm{char} \hspace{3pt} \F$
does not divide $n$, in which case 
there exists an irreducible nilpotent
subgroup of $\GL(n,\F)$ of every 
nilpotency class.
\item[{\rm (ii)}] Irreducible maximal 
locally nilpotent subgroups of $\GL(n,\F)$ 
are monomial and pairwise conjugate.
\end{itemize}
\end{theorem}

By (ii), a completely reducible locally nilpotent 
linear group over an algebraically closed field 
is monomial. The matrix form of the groups in  (ii) 
is discussed in 
\cite[Chapter III, $\S \hspace{.5pt} 7$]{SuprPre}. 

To sum up: although locally nilpotent linear groups 
are well-studied, significant gaps yet remain. Most
 results are concerned with absolutely
irreducible maximal locally nilpotent 
groups, and do not readily
yield analogues for groups that are not maximal 
or not absolutely irreducible (cf.~\cite{Glasgow}). 
Complete classifications of locally nilpotent 
subgroups of $\GL(n,\F)$ are achievable only by 
placing restrictions on $\F$ or $n$. 
In the sequel, we allow arbitrary fields but restrict 
the degree.

\section{Prime degree locally nilpotent linear groups}

In this section, we illustrate how established 
theory of locally nilpotent linear
groups may be applied to obtain a full classification 
in that theory.  Specifically, we classify the 
irreducible maximal 
locally nilpotent subgroups of $\GL(q, \F)$,
where $q$ is prime and $\F$ is any field. This 
classification is in the form of a list
of $\GL(q, \F)$-conjugacy class representatives of 
the groups, with each listed group defined
by a generating set of matrices. We provide 
criteria to decide conjugacy between
listed groups.

Restricting to prime degree $q$ 
offers various advantages. 
An irreducible nonabelian subgroup $G$ of $\GL(q, \F)$
is absolutely irreducible, and is either 
 primitive or monomial.
If $G$ is locally nilpotent then 
$G\hspace{1pt} \F^\times 1_q /\F^\times 1_q$ 
lies in a Sylow $q$-subgroup of $\mathrm{PGL}(q,\F)$, 
whose structure is simpler than that of a Sylow 
subgroup of $\mathrm{PGL}(n,\F)$ for composite 
degree $n$.

A partial classification of the 
irreducible maximal locally nilpotent subgroups
of $\GL(q,\F)$ can be derived from a description
 of the absolutely irreducible maximal 
locally nilpotent subgroups of $\GL(q^a , \F)$. 
However, here we propose novel methods,
and give a complete, self-contained result, 
which may be extended to a complete
classification in degrees $q^a$. 
In particular, we give an exact description
of the Sylow $2$-subgroups of $\PGL(2, \F)$, 
omitted by other authors. This is of
signal importance because 
classifying the Sylow $2$-subgroups of 
$\mathrm{PGL}(n,\F)$ in arbitrary degree $n$ 
depends on having solved the problem for $n=2$ 
(cf.~the case of $2$-subgroups
in \cite{Konyukhpgroups,LGP}).

We first consider absolutely
irreducible groups; abelian groups will be 
dealt with at the end. Note that the methods of 
this section were originally developed in 
\cite{CEJM}, for the purpose of classifying 
 maximal irreducible periodic subgroups of 
$\mathrm{PGL}(q,\F)$. 

By \cite[p. \hspace{-5.5pt}217, Theorem~6]{Supr},  
$\GL(q,\F)$ has absolutely irreducible locally 
nilpotent subgroups if and only if there exists 
an element of order $q$ in $\F^\times$.  
Let $D$ be the set 
$\{\mathrm{diag}(\beta \beta_1 , \allowbreak
\ldots , \beta \beta_q ) 
\hspace{.8pt} \sut \hspace{.8pt} 
\beta_i \in \mathrm{Syl}_q (\F^\times ), 
\hspace{.3pt}  \beta \in \F^\times \hspace{-.7pt} \}$
of diagonal matrices,
where $\mathrm{Syl}_q$ denotes Sylow $q$-subgroup.
For $\alpha \in \F^\times$, define
\[
I_\alpha =  \begin{bmatrix}
0 &  1_{q-1} \\
\alpha &  0 
\end{bmatrix}\in \GL(q,\F)
\]

\vspace{1pt}

\noindent and $I:=I_1$. 
If $H \leq \GL(q,\F)$ then we write
 $\Ddet(H)$ for $\{\det(h) \sut h \in H\}$.

Assuming that $\F^\times$ has an element
of order $q$,
 let $H_\alpha = \langle D, I_\alpha\rangle$.
The subgroup $H_\alpha$ of $\GL(q,\F)$ is monomial 
and absolutely irreducible. Since
 $H_\alpha/\F^\times 1_q$ is a $q$-subgroup of 
$\mathrm{PGL}(q,\F)$, $H_\alpha$ is
locally nilpotent. When $\mathrm{Syl}_q (\F^\times )$
is finite, $H_\alpha$ is nilpotent with 
nilpotency class
$1 + (q - 1) \log_q |\mathrm{Syl}_q (\F^\times )|$.

Denote by $\pi$ the natural epimorphism from the 
group of all monomial matrices in
$\GL(q,\F)$ onto the group $\mathrm{Sym}(q)$ of 
$q \times q$ permutation matrices; 
$\ker \pi$ is the
group $\mathrm{D}(q, \F)$ of all diagonal matrices 
in $\GL(q,\F)$.
\begin{lemma}[Cf.~{\cite[Lemma~22]{CEJM}}]
\label{Lem31}
Let $a, b \in \mathrm{D}(q, \F)$. The following are
equivalent$:$
\begin{itemize}
\item[{\rm (i)}] $Ia$ is $\GL(q,\F)$-conjugate
to $Ib;$
\item[{\rm (ii)}] $Ia$ is $\mathrm{D}(q,\F)$-conjugate
to $Ib;$
\item[{\rm (iii)}] $\det(a) = \det(b)$.
\end{itemize}
\end{lemma}

\begin{lemma}
\label{Lem32}
Let $H$ be an irreducible monomial 
locally nilpotent subgroup of $\GL(q,\F)$.
Then $H$ is conjugate in $\GL(q,\F)$ to a 
subgroup of some $H_\alpha$.
\end{lemma}
\begin{proof}
 If $H$ is abelian then $\pi(H) \leq \mathrm{Sym}(q)$ 
is transitive abelian, i.e., cyclic of
order $q$; thus 
$H \cap \mathrm{D}(q, \F)\leq \F^\times 1_q$. 
If $H$ is absolutely irreducible then
$H\F^\times /\F^\times 1_q$ is a $q$-group, so 
 $H \cap \mathrm{D}(q, \F) \leq D$, and $|\pi(H)|=q$
 again. Hence, up to conjugacy,  
$H \leq \langle D, Ia\rangle$ for some 
$a \in \mathrm{D}(q, \F)$. 
By Lemma~\ref{Lem31}, $H$ is then 
conjugate to a subgroup of $H_{\det(a)}$.
\end{proof}

Denote $\Ddet(D) = 
\mathrm{Syl}_q (\F^\times )(\F^\times )^q$ 
by $S$.
\begin{lemma}\label{Lem33}
If $\alpha \in  S$ then $H_\alpha$ is 
$\mathrm{D}(q, \F)$-conjugate to $H_1$. 
Let $\alpha_1 , \alpha_2 \not \in  S$; then 
$H_{\alpha_1}$ and $H_{\alpha_2}$ are 
$\GL(q,\F)$-conjugate if and only if 
$\hspace{.3pt} 
\Ddet(H_{\alpha_1}) = \Ddet(H_{\alpha_2} )$, i.e., 
$\langle \alpha_1 S \hspace{.8pt} \rangle$  and
$\langle \alpha_2 S  \hspace{.8pt} \rangle$ are 
identical subgroups of $\F^\times /S$ of 
order $q$.
\end{lemma}
\begin{proof}
Suppose that $\alpha =\beta_1 \beta^q$ for
some $\beta_1\in \mathrm{Syl}_q (\F^\times )$ 
and $\beta\in \F^\times$. Then 
$\det(I_\alpha ) = \det(Ib)$ where 
$b = \mathrm{diag}(\beta_1\beta ,\beta , \ldots , \beta) 
\in D$. 
Thus $H_\alpha$ is $\mathrm{D}(q, \F)$-conjugate 
to $\langle Ib, D\rangle = H_1$ by Lemma~\ref{Lem31}.

Suppose that $\alpha_1 , \alpha_2 \not \in  S$
and $\alpha_1 \in \langle \alpha_2, S\rangle$. 
Then $\det(I_{\alpha_1}) = \det(I_{\alpha_2}^r c)$ 
for some $c \in \allowbreak D$ and 
$1 \leq  r \leq  q - 1$. Also, there exists 
$x \in \mathrm{Sym}(q)$ 
such that $x I_{\alpha_2}^r  cx^{-1} = Ib$ for some
$b \in \mathrm{D}(q, \F)$. Hence, by Lemma~\ref{Lem31}
 once more, $H_{\alpha_1}$ and $H_{\alpha_2}$ are
 conjugate (this time by a monomial matrix).
\end{proof}

\begin{corollary}
\label{Cor34}
Define $\mathcal{H} = \{H_\alpha \sut 
\alpha\in \F^\times \setminus S\}$. 
The $\GL(q,\F)$-conjugacy classes of the
groups in $\mathcal H$ are in one-to-one 
correspondence with the distinct subgroups 
of $\F^\times /S$ of
order $q$. 
Therefore the number of such 
classes is finite if and only if 
$\F^\times /S$ is finite.
\end{corollary}

\begin{remark}
 If $\F$ is algebraically closed or 
finite then $\mathcal H=\emptyset$: a 
maximal absolutely irreducible monomial 
locally nilpotent subgroup of $\GL(q,\F)$ 
is conjugate to $H_1$.
\end{remark}

We turn next to primitive groups. 
\begin{lemma}\label{Lem36}
Let $H$ be a primitive locally nilpotent 
subgroup of $\GL(q,\F)$. Then $H$ has
an irreducible abelian normal subgroup.
\end{lemma}
\begin{proof}
Since $H$ is a
locally nilpotent linear group, it is 
solvable. Thus $H$ has an abelian normal 
subgroup $A$ of finite 
 index~\cite[p. \hspace{-3pt}135, Theorem~6]{Supr}. 
If $A \leq \mathsf{Z}(H)$ then 
$H/\mathsf{Z}(H)$ is finite and so $H$
is nilpotent. Recall that a nonabelian 
nilpotent group 
contains a noncentral abelian 
normal subgroup.
Any such subgroup $B$ of $H$ 
must be irreducible. For if it 
were reducible then $B$
would be diagonalizable with 
inequivalent irreducible parts,
contradicting primitivity of $H$.
\end{proof}

By Lemma~\ref{Lem36}, a primitive 
locally nilpotent subgroup of $\GL(q,\F)$ 
is contained in the $\GL(q,\F)$-normalizer 
of the multiplicative group of a field 
extension $\Delta$ of $\F1_q$ of
degree $q$. Since this degree is prime, 
$\Delta$ is a cyclic extension of $\F$, with 
Galois group of order $q$.  
If $\F^\times$ has an element $\xi$
of order $q$, then  $\Delta = \langle h\rangle_\F$
for some $h \in \GL(q,\F)$ such that 
$h^q = \beta 1_q \in \F^\times 1_q$ 
(see \cite[p. \hspace{-3pt}289, Theorem~6.2]{Lang}).
Since  $\beta \in (\F^\times )^q$ implies 
that $h$ is scalar, $\beta = \alpha^r \gamma^q$
for some $\alpha, \gamma\in \F^\times$ and 
$1 \leq r \leq q - 1$. 
Then $\gamma^{-1}h$ and $I_{\alpha}^r$
have the same characteristic (minimal) polynomial
 $X^q -\alpha^r1_q$, which is 
 $\F$-irreducible; so $\gamma^{-1}h$ and 
$I_{\alpha}^r$ are conjugate. 
Hence $\Delta$ is conjugate to $\Delta_\alpha :=
 \langle I_\alpha\rangle_\F$, $\alpha \not \in (\F^\times)^q$.
We see that $\mathsf{N}_{\GL(q,\F)}(\Delta_\alpha^\times)
= \langle \Delta_\alpha^\times, d\rangle$ where
$d=\mathrm{diag}(1,\xi,\ldots, \xi^{q-1})$.
Denote by $G(\alpha, b)$
 the subgroup $\langle A_\alpha, db\rangle$
of $\langle \Delta_{\alpha}^\times, d\rangle$,
where $A_\alpha/\F^\times 1_q$
is the Sylow $q$-subgroup of 
$\Delta_{\alpha}^\times/\F^\times 1_q$
and $b\in \Delta_{\alpha}^\times$.
Since $A_\alpha$ is a noncentral
 irreducible subgroup, $G(\alpha, b)$ is
absolutely irreducible.
\begin{lemma}\label{Lem37}
An absolutely irreducible primitive locally 
nilpotent group $H\leq \GL(q,\F)$
is conjugate to a subgroup of some $G(\alpha, b)$.
\end{lemma}
\begin{proof}
 Up to conjugacy, $H = \langle H \cap \Delta_{\alpha}^\times,
db\rangle$ for some $\alpha \in \F^\times 
\setminus (\F^\times)^q$ and $b\in \Delta_{\alpha}^\times$. 
Then $H \cap \Delta_{\alpha}^\times \leq A_\alpha$
 by Theorem~\ref{Thm24}.
\end{proof}
Let $\epsilon_k$ be an element of 
multiplicative order $2^k$ in $\overline{\F}$. 
We drop the subscript when $k=2$, 
i.e., $\epsilon$ is a square root of $-1$.
\begin{lemma}
\label{Lem38}
Suppose that $\mathrm{char} \hspace{3pt} \F\neq 2$ 
and $\epsilon \not \in \F$. Let $\E = \F(\epsilon)$, and
let $\sigma$ be the $\F$-involution of $\E$. 
If $\mathrm{Syl}_2 (\E^\times ) = \langle \epsilon_m\rangle$
is cyclic then $\mathrm{Syl}_2 (\E^\times /\F^\times )$ 
is cyclic. Explicitly, one of the following must be true$:$
\begin{itemize}
\item[{\rm (i)}] $\sigma(\epsilon_m) = -\epsilon_m^{-1}$ 
and  $\mathrm{Syl}_2 (\E^\times /\F^\times ) = 
\langle \epsilon_m\F^\times\rangle$
has order $2^{m-1};$
\item[{\rm (ii)}] $\sigma(\epsilon_m) = \epsilon_m^{-1}$ 
and  $\mathrm{Syl}_2 (\E^\times /\F^\times ) = 
\langle (1+\epsilon_m)\F^\times\rangle$
has order $2^{m}$.
\end{itemize}
\end{lemma}
\begin{proof}
We make some preliminary comments. Either 
$\sigma(\epsilon_m) = \epsilon_m^{-1}$ or 
$\sigma(\epsilon_m) = -\epsilon_m^{-1}$.
Suppose that $\sigma(\epsilon_k) = \epsilon_k^{-1}$.
Then
\[
\sigma((1+\epsilon_k)^{2^k}) = 
\sigma(1+\epsilon_k)^{2^k} =
(1+\epsilon_{k}^{-1})^{2^k} = 
\left(\frac{1+\epsilon_k}{\epsilon_k}\right)^{2^k}
= (1+\epsilon_k)^{2^k}.  
\]
Thus $(1+\epsilon_k)^{2^k} \in \F$ and 
\begin{equation}\label{Lem38Eq1}
(1 + \epsilon_k )\F^\times \in  
\mathrm{Syl}_2 (\E^\times /\F^\times )
\end{equation}
if $k\geq 2$. Also, if $k > 2$ then 
$\epsilon_{k}^{-1}(\epsilon_{k-1}+1) =
\epsilon_{k}^{-1}(\epsilon_{k}^2+1) = \epsilon_{k}
+\epsilon_{k}^{-1}\in \F^\times$,
 and hence
\begin{equation}\label{Lem38Eq2}
1 + \epsilon_{k-1} \in \langle \epsilon_k\rangle\F^\times .
\end{equation}

Let $x\F^\times$ be a nontrivial element of 
$\mathrm{Syl}_2 (\E^\times /\F^\times )$ 
of order $2^l$; so
$x^{2^l} \in \allowbreak \F^\times \setminus (\F^\times )^2$.

Suppose that $l = 1$. Write $x = a + \epsilon b$
for $a, b \in \F$. 
Then $2ab \epsilon \in \F$ implies that $a = \allowbreak 0$, 
i.e.,
\begin{equation}\label{Lem38Eq3}
x \in \langle \epsilon \F^\times \rangle \leq 
\langle \epsilon_m \F^\times \rangle .
\end{equation}

Suppose next that $l \geq 2$. We have 
$\sigma(x) = yx$ for some $y \in \mathrm{Syl}_2 (\E^\times )$,
$y^{2^l}=\allowbreak 1$. Now $y \neq -1$,
because if $y=-1$ then $|x\F^\times | = 2 < 2^l$. 
Further, $x = \sigma^2 (x) = \sigma(yx) = \allowbreak
 \sigma(y)yx$ and so
\begin{equation}\label{Lem38Eq4}
\sigma(y) = y^{-1} .
\end{equation}
Since $\mathrm{tr}(x) = (1 + y)x \in \F$,
\begin{equation}\label{Lem38Eq5}
x \in (1+y)^{-1}\F^\times =  
(1+\sigma(y))\F^\times = (1+y^{-1})\F^\times.
\end{equation}

We are ready to complete the proof that 
 (i) or (ii) must be true. Let
$\sigma(\epsilon_m) = -\epsilon_m^{-1}$, so 
 $m > 2$. By \eqref{Lem38Eq3}, we may take 
$l \geq 2$, in which event $4 < |y| \leq  2^{m-1}$
by \eqref{Lem38Eq4}.
Then $(1 + y^{-1} )\F^\times 
\subseteq \langle \epsilon_m\rangle \F^\times$ 
by \eqref{Lem38Eq2}, and by \eqref{Lem38Eq5}, 
this yields (i).

Let $\sigma(\epsilon_m ) = \epsilon_m^{-1}$.
By \eqref{Lem38Eq1},
$\langle (1+\epsilon_m)\F^\times\rangle
\subseteq \mathrm{Syl}_2 (\E^\times /\F^\times )$.
As $\sigma$ fixes $\epsilon_m^{-1} (1 + \epsilon_m )^2$,
additionally $\epsilon_m\F^\times \in \allowbreak 
\langle (1+\epsilon_m)\F^\times\rangle^2$.
 Part (ii) follows from \eqref{Lem38Eq2}, 
\eqref{Lem38Eq3}, and \eqref{Lem38Eq5}.
\end{proof}

\begin{corollary}
\label{Cor39}
If $\mathrm{Syl}_2 (\E^\times )$ is quasicyclic 
then $\mathrm{Syl}_2 (\E^\times /\F^\times )$ is
quasicyclic, and equal to
$\{\langle \epsilon_k \F^\times \rangle 
\sut \epsilon_k \allowbreak \in \mathrm{Syl}_2 (\E^\times )\}$.
\end{corollary}
\begin{proof}
Since $\sigma(\epsilon_k) =\epsilon_{k}^{-1}$ 
for each $\epsilon_k \in \mathrm{Syl}_2 (\E^\times )$,  
we cannot have $\sigma(\epsilon_k) =-\epsilon_{k}^{-1}$. 
The corollary is then a consequence of 
Lemma~\ref{Lem38}~(ii) and \eqref{Lem38Eq2}. 
\end{proof}

\begin{lemma}\label{Lem310}
Let $|\E : \F| = q$ and $\E = \F(a)$, 
where $a^{q^l}\in\F$. 
Suppose that $\F^\times$ has an element 
of order $q$, and $\E \neq  \F(\epsilon)$ 
if $q = 2$. 
Then $\mathrm{Syl}_q (\E^\times /\F^\times ) 
= \langle a\F^\times \rangle$.
\end{lemma}
\begin{proof}
If $q > 2$, or $q = 2$ and $\epsilon\in \F$, 
then the result follows from 
\cite[Lemma~2]{Konyukh8}.

Let $q = 2$. Select $x\F^\times\in 
\E^\times/\F^\times$ of order $2^m \geq 2$; 
say $x^{2^m}=\alpha\in 
\allowbreak \F^\times \setminus (\F^\times )^2$. 
If $\alpha = -4\gamma^4$ for $\gamma\in \F^\times$, 
then $\gamma^{2^{m-1}}/2\gamma^2$
is a square root of $-1$, contradicting
$\E \neq  \F(\epsilon)$.

Suppose that $\alpha \not \in -4(\F^\times )^4$. 
The polynomial $X^4-\alpha$ is $\F$-irreducible,
so that if $m\geq \allowbreak 2$ then $|\E : \F | \geq 4$. 
Hence $m = 1$, $a = \sqrt{\beta}$ 
for some $\beta \in \F^\times$, and $x =\sqrt{\alpha}$. 
Now $\sqrt{\alpha} = \allowbreak x_1+x_2\sqrt{\beta}$
for some $x_1, x_2\in \F$. Then 
$\alpha = x_1^2 +\beta x_2^2 +2x_1x_2\sqrt{\beta}$
implies that $x_1 = 0$ or $x_2 = 0$. As the 
latter is impossible, 
 $x \in \langle a\F^\times 1_2\rangle$.
\end{proof}

\begin{lemma}\label{Lem311}
If $q > 2$ or $\alpha \not \in -(\F^\times)^2$
then $A_\alpha$ is the monomial group 
$\langle I_\alpha, \F^\times 1_q\rangle$.
Otherwise, $A_\alpha$ is primitive.
\end{lemma}
\begin{proof}
In Lemma~\ref{Lem310}, put $\E = \Delta_\alpha$
and $a = I_\alpha$.  Then 
$A_\alpha = \langle I_\alpha, \F^\times 1_q \rangle$
 unless $q = 2$ and $\Delta_\alpha \cong \F(\epsilon)$,
 i.e., $\alpha \in -(\F^\times )^2$. 
If $A_{-\gamma^2}\leq \GL(2,\F)$ were monomial 
then its square would be scalar; however,
$1_2 +\gamma^{-1} I _{-\gamma^2}
\in A_{-\gamma^2}$ has projective order $4$.
\end{proof}

We refer to the set of hypotheses $q = 2$ and 
$\alpha \in -(\F^\times)^2$ as case $(*)$. 
Lemma~\ref{Lem38} and Corollary~\ref{Cor39}
 describe the $A_\alpha$ in case $(*)$. 
Actually, a group $G(-\gamma^2 , b)$ in this 
case is conjugate to some $G(-1, b' )$, 
since $I_{-\gamma^2}$ is 
$\mathrm{D}(q, \F)$-conjugate to
$\gamma I_{-1}$ by Lemma~\ref{Lem31}. 
In all but case $(*)$, 
$|G(\alpha, b)/\F^\times 1_q | = q^2$
 (because $(db)^q = \det(b)1_q$
and $[I_\alpha , d]$ is scalar), so $G(\alpha, b)$ 
is class $2$ nilpotent. The locally nilpotent
group $G(-1, b)$ is nilpotent only if 
$\mathrm{Syl}_2 (\Delta_{-1}^\times)$ is cyclic, 
when $G(-1, b)/\F^\times 1_q$ is a dihedral
$2$-group and $G(-1, b)$ has nilpotency class 
$\log_2 |\mathrm{Syl}_2 (\Delta_{-1}^\times)\F^\times 1_2|$.

\begin{lemma}
\label{Lem312}
In case $(*)$, $G(\alpha, b)$ is primitive. 
In all other cases, $G(\alpha, b)$ is primitive
if and only if $\det(b) \not \in 
\langle (-1)^{q-1}\alpha, (\F^\times )^q\rangle 
= \Ddet(A_\alpha)$.
\end{lemma}
\begin{proof}
By Lemma~\ref{Lem311}, assume that 
we are not in case $(*)$. 
According to \cite[Lemma~1]{Konyukh8},
$G(\alpha, b)$ is primitive 
if and only if its elements of order $q$ 
are all scalar.
Suppose that $\det(b) \not \in \Ddet(A_\alpha )$
 and let $h \in G(\alpha, b)$, $|h| = q$. 
If $h \not \in  A_\alpha$, i.e., 
$h = \allowbreak dbb_1$
for some $b_1 \in A_\alpha$, then 
 $h^q = \det(bb_1 )1_q$ implies that 
$\det(b) \in \Ddet(A_\alpha )$. 
Thus $h \in A_\alpha$, and $h$ is
scalar by Lemma~\ref{Lem311}. 
Conversely, if $\det(b) \in \Ddet(A_\alpha )$
 then $dbx$ for some $x\in A_\alpha$ is a
nonscalar element of $G(\alpha, b)$ of order $q$.
\end{proof}
\begin{remark}
 Except in case $(*)$, if $\F$ is finite then 
$G(\alpha, b)$ is monomial.
\end{remark}
\begin{lemma}\label{Lem314} 
For $i = 1, 2$, let $g_i = db_i$ where 
$b_i \in \Delta_\alpha$. The following  are
equivalent$:$
\begin{itemize}
\item[{\rm (i)}] 
 $g_1$ is $\GL(q,\F)$-conjugate to $g_2;$
\item[{\rm (ii)}] 
 $g_1$ is $\Delta_\alpha$-conjugate to $g_2;$
\item[{\rm (iii)}] $\det(b_1 ) = \det(b_2 )$.
\end{itemize}
\end{lemma}
\begin{proof} See \cite[Lemma~23]{CEJM}.
\end{proof}

\begin{corollary}
\label{Cor315}
Primitive groups 
$G(\alpha, b_1 )$, $G(\alpha, b_2 )$ 
not in case $(*)$ are 
conjugate if and only if 
$\Ddet(G(\alpha, b_1 )) = \Ddet(G(\alpha, b_2 ))$
 and $\det(b_1 ) = \det(b_2 c)$ for some 
$c \in A_\alpha$.
\end{corollary}
\begin{proof}
 Suppose that $tG(\alpha, b_1 )t^{-1} = G(\alpha, b_2 )$.
 Since $t$ normalizes $A_\alpha$ (Lemma~\ref{Lem312}),
we have $t\in \langle d, \Delta_{\alpha}^\times\rangle$. 
Then it may be checked that $tdb_1 t^{-1}\in db_2A_\alpha$.
The other direction is clear by Lemma~\ref{Lem314}.
\end{proof}

Denote by $\mathcal G$ the set of all primitive 
$G(\alpha, b)$ such
that the only case $(*)$ groups in
 $\mathcal G$ are the $G(-1, b)$.

\begin{remark} $\mathcal{G}=\emptyset$ if
 $\F$ is algebraically closed, for then 
$G(\alpha, b)$ is not defined. When $\F$ 
is finite, $\mathcal G\neq \emptyset$ if and 
only if $q = 2$ and $|\F| \equiv 3 \pmod 4$.
\end{remark}

\begin{lemma}\label{Lem317}
The subset 
$\widetilde{\mathcal{G}}$ of
of $\mathcal G$ consisting of all groups not 
in case $(*)$ splits up into finitely many 
$\GL(q,\F)$-conjugacy classes if and only if 
$\F^\times /(\F^\times )^q$ is finite.
\end{lemma}
\begin{proof}
If the number of conjugacy classes
in $\widetilde{\mathcal{G}}$ is finite,
then $\F^\times /(\F^\times )^q$ 
is finitely generated and hence finite. 
Conversely, if there are only finitely many 
candidates for $\Ddet(G(\alpha, b))$, 
 then the number of conjugacy classes in 
 $\widetilde{\mathcal{G}}$ 
is finite by Corollary~\ref{Cor315}.
\end{proof}

We now present the main classification.
\begin{theorem}
\label{Thm318}
 Suppose that $\F^\times$ has an element of order $q$. 
A subgroup of $\GL(q,\F)$ is an absolutely irreducible
maximal locally nilpotent subgroup of $\GL(q,\F)$ 
if and only if it is conjugate to a group
in $\mathcal{N} = \{H_1 \} \cup \mathcal{H}\cup \mathcal{G}$, 
with the following exceptions when $q = 2$ and 
$\epsilon \not \in \F \! :$
\begin{itemize}
\item[{\rm (i)}]
 $H_1$ is a proper subgroup of $G(-1, 1) \in \mathcal{G};$
\item[{\rm (ii)}] 
if $\alpha\not \in -(\F^\times )^2$ and 
either $\det(b) \in -(\F^\times )^2$ or 
$\det(b) \in \alpha(\F^\times )^2$, then $G(\alpha, b)$ 
is conjugate to a proper subgroup of $G(-1, c)$ 
where $\det(c) = \alpha$.
\end{itemize}
\end{theorem}
\begin{proof}
We observed previously that all groups 
in $\mathcal N$ are absolutely irreducible
locally nilpotent. By Lemmas~\ref{Lem32}, \ref{Lem33}, 
\ref{Lem37}, and remarks after Lemma~\ref{Lem311}, an
absolutely irreducible locally nilpotent subgroup 
of $\GL(q,\F)$ is conjugate to a subgroup of a group in 
$\mathcal N$. 
It remains to show that the 
$H_\alpha$ and $G(\alpha, b) \in\mathcal{G}$ 
are  maximal locally nilpotent, 
with the stated exceptions.

Let $G$ be a maximal locally nilpotent subgroup 
of $\GL(q,\F)$ containing $H_\alpha$. If $G$
is monomial then $tGt^{-1} = H_\beta$  for some
$t$, $H_\beta$.
If $tDt^{-1} \neq D$ then
$tDt^{-1}\cap D$ is scalar of index $q$ in $D$, 
so  $|H_\beta /\F^\times 1_q | = q^2$ ; but 
$H_\beta /\F^\times 1_q$ has cardinality
at least $q^{q+1}$. 
Thus $tDt^{-1} = D$, and then $q = |H_\beta : D| \geq
 |tH_\alpha t^{-1} : D| = |H_\alpha : D| = q$.
Therefore $tH_\alpha t^{-1} = H_\beta$, i.e., 
$H_\alpha = G$. 
Next suppose that $G$ is primitive, hence conjugate
to some $G(\alpha_1 , b)$. In every case other than 
$(*)$, $|G/\F^\times 1_q | = q^2$ 
is less than the cardinality of $H_\alpha /\F^\times 1_q$.
 Hence $q = 2$, $\epsilon \not \in \F^\times$, 
$G$ is conjugate to $G(-1, b)$, and
$H_\alpha = \langle d, I_\alpha , 
\F^\times 1_2 \rangle$. Either $I_\alpha$
 or $I_{-\alpha} = dI_\alpha$ is 
conjugate to $h \in A_{-1}$ such
that $h^2 \in \F^\times 1_2$. Now $h$ has the form 
$\eta I_{-1}$, $\eta \in \F^\times$, and by 
comparing determinants we get $\alpha = \pm \eta^2$.
Thus if $H_\alpha \in \mathcal{H}$ then $H_\alpha$ 
is maximal. However, 
$H_1 = \langle d, I_{-1} , \F^\times 1_2\rangle$
is a proper subgroup of $G(-1, 1)$.

Let $G$ be a maximal locally nilpotent subgroup 
of $\GL(q,\F)$ containing $G(\alpha, b) 
\in \allowbreak \mathcal{G}$.
Then $tGt^{-1} = G(\alpha_1 , b_1 ) \in\mathcal{G}$
for some $t$.
 Apart from when $q = 2$, $\epsilon\not \in \F^\times$,
and $\alpha_1 = \allowbreak -1$, we have 
$|G(\alpha, b)/\F^\times 1_q | = 
|G(\alpha_1, b_1)/\F^\times 1_q | = 
|G/\F^\times 1_q | = q^2$, 
and thus $G(\alpha, b) = \allowbreak G$. 
Suppose that $q = 2$, $\epsilon \not \in \F$, and 
$\alpha_1 = \allowbreak -1$. 
If $\alpha \in -(\F^\times )^2$ then
$G(\alpha, b)$ is conjugate to some $G(-1, b_2 )$, 
so that  $G(\alpha, b) =
\allowbreak  G$. Therefore, if $G(\alpha, b)$ is
not maximal then $\alpha \not \in -(\F^\times )^2$. 
For the rest of the proof, 
$\alpha \not \in -(\F^\times )^2$, which means that 
$G(\alpha, b) = \langle I_\alpha , db, 
\F^\times 1_2\rangle$. One of the 
following occurs: $tI_\alpha t^{-1} \not \in A_{-1}$
 and $tdbt^{-1}\in  A_{-1}$, or
 $tI_\alpha t^{-1} \not \in A_{-1}$ and 
$tdbt^{-1}\not \in  A_{-1}$.
In the first situation, $\det(b) \in -(\F^\times )^2$;
in the second, $[I_\alpha , db]$ scalar forces 
$tdbt^{-1} \in  I_{-1} tI_\alpha t^{-1} \F^\times$,
so $\det(b) \in \alpha (\F^\times )^2$. 
Suppose that $\det(b) \in \allowbreak -(\F^\times )^2$. 
Then $\det(db) = \allowbreak \det(\eta I_{-1} )$ 
for some $\eta \in \F^\times$. Since $db$ and 
$\eta I_{-1}$ have zero trace,
these elements are conjugate, say 
$sdbs^{-1} = \eta I_{-1}$. Also, 
$sI_\alpha s^{-1}$ normalizes $\Delta_{-1}^\times$, 
because $I_\alpha db I_{\alpha}^{-1} = -db$.
Thus $sI_{\alpha} s^{-1} = dc$,
where $c \in \Delta_{-1}^\times$ 
and $\det(c) = \alpha$. It follows that 
$G(\alpha, b)$ is conjugate to a proper subgroup
of $G(-1, c)$. Entirely similar reasoning leads to
this same conclusion when 
$\det(b) \in \alpha(\F^\times )^2$.
\end{proof}

The next lemma completes our classification.
\begin{lemma}
\label{Lem319}
Suppose that $\GL(q,\F)$ has irreducible 
abelian subgroups, and let $H$ be an
irreducible maximal abelian subgroup of $\GL(q,\F)$.
\begin{itemize}
\item[{\rm (i)}] If $\hspace{.4pt} \F^\times$ has no 
element of order $q$, then $H$ is a maximal locally 
nilpotent subgroup of $\GL(q,\F)$, and
any maximal locally nilpotent subgroup 
of $\GL(q,\F)$ is abelian.
\item[{\rm (ii)}] 
Suppose that $\F^\times$ has an element of 
order $q$. Then $H$ is a maximal 
locally nilpotent subgroup of $\GL(q,\F)$ unless
$q = 2$ and $\epsilon \not \in \F;$ in that case,
$H$  is a maximal locally nilpotent
subgroup if and only if 
$H/\F^\times 1_2$  is not a $2$-group.
\end{itemize}
\end{lemma}
\begin{proof}
We prove (ii). Clearly $H$ is not monomial. 
By Lemma~\ref{Lem311}, 
if $H$ is not maximal locally nilpotent 
then $q = 2$, $\epsilon\not \in \F$, and 
$H = A_\alpha\leq G(\alpha, \beta) \in \mathcal{G}$. 
\end{proof}

\bigskip

\bibliographystyle{amsplain}

\begin{thebibliography}{10}

\bibitem{Beals}
R.~Beals, \emph{Algorithms for matrix groups 
 and the Tits alternative}, 
J. Comput. System Sci.~\textbf{58} (1999), 260--279.

\bibitem{Glasgow}
A.~S. Detinko and D.~L. Flannery,
\emph{Classification of nilpotent primitive 
 linear groups over finite fields},
Glasg. Math. J. \textbf{46} (2004), 
no.~3, 585--594.

\bibitem{CEJM}
A.~S. Detinko and D.~L. Flannery, 
\emph{Periodic subgroups of projective 
linear groups in positive characteristic},
  Cent. Eur. J. Math. \textbf{6} (2008), 
	no.~3, 384--392.

\bibitem{Gromov} M.~L.~Gromov, 
\emph{Groups of polynomial growth 
 and expanding maps}, 
 Inst. Hautes \'{E}tudes Sci. 
Publ. Math. No. 53 
(1981), 53--73.

\bibitem{Holt}
D.~F. Holt, B.~Eick, and E.~A. O'Brien,
 \emph{Handbook of computational group theory},
Chapman \& Hall/CRC, Boca Raton, 2005.

\bibitem{KonyukhPGL}
V.~S. Konyukh, \emph{Sylow $p$-subgroups 
of a projective linear group}, 
Vests\={\i} Akad. Navuk BSSR 
Ser. F\={\i}z.-Mat. Navuk (1985), 
no. 6, 23--29, 124--125.

\bibitem{Konyukhpgroups}
V.~S. Konyukh, \emph{On linear $p$-groups}, 
Vests\={\i} Akad. Navuk BSSR 
Ser. F\={\i}z.-Mat. Navuk (1987), 
no.~1, 3--8, 124.

\bibitem{Konyukh8}
V.~S.~Konyukh,
 \emph{Irreducible locally nilpotent 
linear groups}, 
Fundam. Prikl. Mat. \textbf{4} (1998), 
1345--1364.

\bibitem{Lang}
S.~Lang, \emph{Algebra}, 
Graduate Texts in Mathematics \textbf{211},
Springer-Verlag, New York, 2002.

\bibitem{LGP}
C.~R.~Leedham-Green and W.~Plesken, 
\emph{Some remarks on Sylow subgroups 
 of general linear groups},
Math. Z. \textbf{191} (1986), 529--535.

\bibitem{Segal}
D.~Segal, \emph{Polycyclic groups}, 
Cambridge University Press,
Cambridge, 1983.

\bibitem{ShirvaniWehrfritz}
M.~Shirvani and B.~A.~F. Wehrfritz, 
\emph{Skew linear groups}, 
London Mathematical Society Lecture Note
Series, vol. 118, Cambridge University Press, 
Cambridge, 1986.

\bibitem{SuprPre}
D.~A. Suprunenko,
\emph{Soluble and nilpotent linear groups}, 
Transl. Math. Monogr., vol. 9, American
Mathematical Society, Providence, RI, 1963.

\bibitem{Supr}
D.~A. Suprunenko, \emph{Matrix groups}, 
Transl. Math. Monogr.,
vol. 45, American Mathematical Society, 
Providence, RI, 1976.

\bibitem{Wehrfritz}
B.~A.~F. Wehrfritz, 
\emph{Infinite linear groups},
Springer-Verlag, New York, 1973.

\end{thebibliography}

\end{document}